\documentclass[12pt]{amsart}

\usepackage{stmaryrd}

\usepackage{cite}

\usepackage{amsmath,amssymb}
\usepackage{tikz}
\usetikzlibrary{patterns}

\usepackage{amsfonts,color}
\usepackage{latexsym,amssymb}
\usepackage{amsmath}
\usepackage[utf8]{inputenc}
\usepackage{enumitem}

\newtheorem{theorem}{Theorem}[section]
\newtheorem{lemma}[theorem]{Lemma}
\newtheorem{proposition}[theorem]{Proposition}

\newcommand{\R}{\mathbb{R}}

\usepackage[bookmarksnumbered,colorlinks, linkcolor=blue, citecolor=red, bookmarks, breaklinks]{hyperref}

\catcode`@=11 \@addtoreset{equation}{section} \catcode`@=12

\usepackage[left=3cm,right=3cm,top=3cm,bottom=3cm]{geometry}

\begin{document}

\title[On Grushin--Schrödinger equation in $\R^N$ ]{On semilinear Grushin--Schrödinger equation in $\R^N$}

\author[J. Carvalho]{J. Carvalho}
\address{Universidade Federal de Sergipe, Departamento de Matem\'atica, 49100-000 São Cristóvão-SE, Brazil}
\email{jonison@mat.ufs.br}

\author[A. Viana]{A. Viana}
\address{Universidade Federal de Sergipe, Departamento de Matem\'atica, 49100-000 São Cristóvão-SE, Brazil}  
\address{Universit\"at Ulm, Institut f\"ur Angewandte Mathematik, 89069, Ulm-BW, Germany}
\email{arlucioviana@academico.ufs.br}



\begin{abstract} 
	We establish the existence of nontrivial nonnegative  weak solutions to the following equation  
	\begin{equation*}
		-\Delta_\gamma u + V(z)u = Q(z)f(u), \quad z\in \mathbb{R}^N, 
	\end{equation*}
	where $\Delta_\gamma $ denotes the so-called Grushin-type operator in $\mathbb{R}^N$. The potentials $V$ and $Q$ are assumed to be controlled below and above, respectively, by functions of type $(1+|z|)^a$,  $a\in\R$. The main result is the embedded of the space $E_V^\gamma$  into the weighted Lebesgue space $L_Q^p(\R^N)$, under suitable conditions. Finally, we derive regularity results for the obtained weak solutions.

\end{abstract}

\maketitle

\section{Introduction and main results}

Consider the  Grushin-type operator $\Delta_\gamma$ defined by
\begin{equation}\label{grushin}
	\Delta_\gamma u (z) := \Delta_x u(z) + |x|^{2\gamma}\Delta_y u(z),
\end{equation}
where $\gamma>0$, $z=(x,y) \in \mathbb{R}^m\times \mathbb{R}^k$, with $m+k = N\geq3$, and $\Delta_x$ and $\Delta_y$ denote the classical Laplacians with respect to the variables $x$ and $y$, respectively. The operator defined in \eqref{grushin} was introduced in \cite{Grushin,grushin-71}. 
Such operators naturally arise in degenerate and anisotropic diffusion problems. 
For integer values of $\gamma$, these operators belong to the general class of  H\"ormander operators \cite{Hormander-67}. For positive real values of $\gamma$, they represent the simplest example of the $\Delta_\lambda$-operators introduced by Kogoj and Lanconelli in  \cite{Kogoj-Lan-12}, and they also fall into the broader family of $X$-elliptic operators studied in \cite{Lan-Kog-00}. Franchi and Lanconelli were the first to introduce and study a metric and the associated underlying sub-Riemannian structure for these types of operators in their seminal works of the early 1980s \cite{Fr-La-83, Fr-La-84}.

Degenerate second-order elliptic operators have remained an object of interest in recent years, see e.g. \cite{Ab-Fe-Luz-25, Ba-Fur-I-15, Biagi-Monticelli-Punzo}. On the other hand, the study of semilinear equations involving these type of operators is still developing, in particular, there are relatively few results concerning semilinear elliptic equations with  Grushin-type operators, see e.g. \cite{Alves-G-Lo-24,Alves-deH-24,B-Radu-21,Du-Ph-17, La-Pa-02,LTW-20, Loi-19}. 

We mention two recent contributions in this direction. In \cite{Loi-19}, Loiudice provides the asymptotic behavior of solutions - both near the singularity and at infinity - for the Dirichlet problem of the form
$$
\begin{cases}-\Delta_\gamma u-\mu \dfrac{\psi^2}{d^2} u=K(z)|u|^{2^*_\gamma-2} u ,& \text { in } \Omega ,\\ u=0 ,& \text { on } \partial \Omega,\end{cases} 
$$
with $2^*_\gamma=\frac{2 N_\gamma}{N_\gamma-2}$, $N_\gamma=m+(\gamma+1) k$, and $\psi:=\left|\nabla_\gamma d\right|$, where
$$d(z)=\left(|x|^{2(\gamma+1)}+(\gamma+1)^2|y|^2\right)^{\frac{1}{2(\gamma+1)}} \mbox{\quad and\quad } \nabla_\gamma d = (\nabla_x d , |x|^\gamma \nabla_y d).$$
Here \(\Omega\) is an open subset of \(\mathbb{R}^N\), \(0\in \Omega\), $K \in L^{\infty}(\Omega)$, and $  0 \leq \mu<\left(\frac{N_\gamma-2}{2}\right)^2$.

In \cite{Alves-Angelo}, Alves and de Holanda proved the 
existence of nontrivial nonnegative solutions to the semilinear elliptic equation
\begin{equation}\label{grushinC}
	-\Delta _{\gamma }u+a(z) u=f(u), \quad z\in \mathbb{R}^N,
\end{equation}
with \(N \geq3\), under the following assumptions on the potential $a$:
\begin{enumerate}
    \item $a(z) \geq a_0>0, \text { for all } z \in \mathbb{R}^N$;
    \item $
a(x, y)=a\left(x^{\prime}, y\right)$ whenever $|x|=\left|x^{\prime}\right|$;
\item either the function $a(x,y)$ periodic in the $y$-variable $y$: $$ a(x, y)=a\left(x, y+y_0\right), \text { for all } y_0 \in \mathbb{Z}^k,
$$
or coercive in $y$, that is,
$$
a(x, y) \longrightarrow \infty, \text { as }|y| \longrightarrow \infty, \text { uniformly for } x \in \mathbb{R}^m .
$$
\end{enumerate}
Furthermore, the authors also consider the case \(a(z)=0\), known as the zero mass case. 

The main goal of the present work is to establish the compact embedding $$E_V^\gamma \hookrightarrow L^p_Q(\mathbb{R}^N) ,$$ under suitable conditions on $V$ and $Q$. As consequence, we prove the existence of nonnegative weak solutions to the equation
\begin{equation}\tag{$\mathcal{P}$}\label{P}
	-\Delta_\gamma u + V(z)u = Q(z)f(u), \quad z\in \mathbb{R}^N.
\end{equation}
In this work, we assume instead the following assumptions on the potentials:
\begin{enumerate}[label=($V$),ref=$(V)$]
	\item \label{V} $V: \mathbb{R}^N\to \mathbb{R}$ is a measurable function and there exist constants  $C_V>0$ and $\alpha\in\mathbb{R}$ such that
	$$
	\frac{C_V}{\left(1+|z|\right)^\alpha}\leq V(z), \quad\mbox{for a.e. } z\in\mathbb{R}^N;
	$$
\end{enumerate}

\begin{enumerate}[label=($Q$),ref=$(Q)$]
	\item \label{Q} $Q: \mathbb{R}^N\to \mathbb{R}$ is a measurable function and there there exist constants   $C_Q>0$ and $\beta\in\mathbb{R}$ such that
	$$
	0<Q(z)\leq \frac{C_Q}{\left(1+|z|\right)^\beta}, \quad\mbox{for a.e. } z\in\mathbb{R}^N.
	$$
\end{enumerate}

Our problem also differs from \eqref{grushinC} due to the presence of the potential $Q(z)$ multiplying $f(u)$; besides that, we adopt distinct assumptions on $V$, without requiring any symmetry, which prevents us from following Strauss's approach  \cite{Strauss} by proving a Radial Lemma, as employed in \cite{Su-Wa-Wi-07}. Instead, we decompose the exterior domain into annular regions - a technique similar to the one presented in \cite{OK-90} and successfully applied in \cite{Carvalho-Furtado-Medeiros-AMPA,Car-Fu-Me-22}. The same technique was used by Ambrosetti \textit{et al.} \cite{Amb-Fel-Mal-04} for a similar class of nonlinear Schrödinger equations involving the Laplacian rather than the Grushin operator. To apply this technique, we prove an auxiliary Lemma which relies on the Sobolev embeddings obtained by Kogoj and Lanconelli \cite{Kogoj-Lan-12}: given a bounded domain $\Omega \subset \R^N$ and $p\in [1,2^*_\gamma]$, then
\begin{equation*}
H_0^{1,\gamma}(\Omega) \hookrightarrow L^p(\Omega).
\end{equation*}
Here, 
$$H_0^{1,\gamma}(\Omega) :=  \overline{C_0^1(\Omega)}^{\| \cdot \|_{H^{1,\gamma}}},$$
and
\begin{equation}\label{norm}
	\|u\|_{H^{1,\gamma}} :=  \left(\int_{\Omega} |\nabla_\gamma u|^2\ \,\mathrm{d}z\right)^{\frac{1}{2}} . 
\end{equation}
As Sobolev embeddings are naturally of interest, we mention that, in the case of the whole space, Monti and Morbidelli \cite{Mo-Mo-06} write that the Sobolev embedding is a consequence of the Poincaré inequality and the representation formula proved in \cite{Fr-La-83} and \cite{FLW-96}, respectively:
\begin{equation}\label{monti}
    H^{1,\gamma}\left(\mathbb{R}^N\right) \hookrightarrow L^{2_\gamma^*}\left(\mathbb{R}^N\right) ,
\end{equation}
where $2_\gamma^*=\frac{2 N_\gamma}{N_\gamma-2}$. See also \cite{Monti06}. More recently, Abatangelo, Ferrero and Luzzini \cite{Ab-Fe-Luz-25} studied the asymptotic behaviors and unique continuation principles for weak solutions to the linear Grushin equation with potential in bounded domains, that is, \eqref{P}, with $Q\equiv0$. They also provide Sobolev embeddings under certain conditions.

Next, we make precise the spaces we are going to work with. We consider $E^\gamma_V$ as the completion of $C_0^\infty(\mathbb{R}^N)$ with respect to the norm
$$
\|u\|_{E^\gamma_V} := \left(\int_{\mathbb{R}^N}\left[|\nabla_\gamma u|^2 + V(z)u^2\right]\,\mathrm{d}z \right)^{\frac{1}{2}}.
$$
This norm is induced by the following inner product:
\[
\left\langle u, v\right\rangle_{E_V^\gamma}:=\int_{\mathbb{R}^N}\left(\nabla_\gamma u \cdot \nabla_\gamma v + V(z)uv\right)\,\mathrm{d}z,
\]
and hence $E^\gamma_V$ is a Hilbert space.

For $1\leq p<\infty$, we define
$$
L^p_Q(\mathbb{R}^N):=\left\lbrace u:\mathbb{R}^N \to \mathbb{R} \   \mbox{measurable} : \int_{\mathbb{R}^N}Q(z)|u|^p\,\mathrm{d}z<\infty \right\rbrace, 
$$
with respect to the norm
$$
\|u\|_{L^p_Q(\mathbb{R}^N)}:=\left(\int_{\mathbb{R}^N}Q(z)|u|^p\,\mathrm{d}z\right)^{\frac{1}{p}}. 
$$

Our main results are the following.

\begin{theorem}[Weighted Sobolev Embedding]
	\label{imersao} Assume that \ref{V} and \ref{Q} hold. Then the embedding $E_V^\gamma \hookrightarrow L^p_Q(\mathbb{R}^N)$ is continuous  if one of the following conditions holds: 	
	\begin{itemize}
		\item [(i)] $\alpha \leq  N < \beta$ and $2\leq p \leq 2_\gamma^*$;
		
		\item[(ii)] $N
        \leq\alpha < \beta$ and $2\leq p \leq \overline{p}$,
		where
		$$
		\overline{p}:=\min\left\lbrace \frac{2(\beta-N)}{(\alpha-N)},2_\gamma^* \right\rbrace;
		$$
		
		\item [(iii)] $\alpha\leq N$, $ \underline{\beta} \leq \beta <N$, and $\underline{p}\leq p\leq2_\gamma^*$,
		where
		$$
		\underline{\beta}:=\max\left\lbrace N-\frac{2_\gamma^*(N-2)}{2},N-\frac{2_\gamma^*(N-\alpha)}{2}\right\rbrace 
		$$
		and
		$$
		\underline{p}:=\max\left\lbrace 2, \frac{2(N-\beta)}{(N-2)}, \frac{2(N-\beta)}{(N-\alpha)} \right\rbrace.
		$$

	\end{itemize}
	Furthermore, the embedding is compact whenever:
	\begin{itemize}
		\item [(i')] $\alpha \leq  N < \beta$ and $2\leq p < 2_\gamma^*$;
		
		\item[(ii')] $N<\alpha\leq\beta$ and $2\leq p < \overline{p}$.

	\end{itemize}
\end{theorem}

We say that a function $u\in E_V^\gamma$ is a weak solution to problem \eqref{P}, if for any $\varphi \in C_0^\infty(\mathbb{R}^N)$, it holds that
$$
\int_{\mathbb{R}^N}\left(\nabla_\gamma u \cdot \nabla_\gamma \varphi + V(z)u\varphi\right)\,\mathrm{d}z = \int_{\mathbb{R}^N}Q(z)f(u)\varphi\,\mathrm{d}z.
$$

Theorem \ref{imersao} helps us to prove the existence of nontrivial nonnegative weak solutions, which is Theorem \ref{existencia_de_solucao} below. In its proof, we will consider only the case $\alpha\leq N < \beta$, since the other cases are similar. Moreover, \(f : \mathbb{R} \to \mathbb{R}\) is continuous and we shall assume that

\vspace{0,2cm}

\begin{enumerate}[label=($f_1$),ref=$(f_1)$]
	\item \label{(f_1)} 
    \(\displaystyle\lim_{s\to 0}\dfrac{f(s)}{s}=0\);
\end{enumerate}

\begin{enumerate}[label=($f_2$),ref=$(f_2)$]
	\item \label{(f_2)} there exists $q \in\left(2,2_\gamma^*\right)$ such that	
	$$
	\limsup _{|s| \rightarrow\infty} \frac{|f(s)|}{|s|^{q-1}}<\infty;
	$$
\end{enumerate}

\begin{enumerate}[label=($f_3$),ref=$(f_3)$]
	\item \label{(f_3)} there exists $\theta>2$ such that
	$$
	0<\theta F(s) \leq f(s) s, \quad \forall \, s \in \mathbb{R} \backslash\{0\},
	$$	
	where $F(t):=\int_0^t f(s) \,\mathrm{d} s$.
\end{enumerate}

Before stating the next result, we define the space 
$$W_\gamma^{2, p}(\R^N)  =\left\{u \in L^p\left(\mathbb{R}^N\right): |x|^{|\beta|\gamma}D^\alpha_x D^\beta_yu \in L^p\left(\mathbb{R}^N\right), \ \mbox{for multi-index} \ |\alpha|+|\beta|\leq 2\right\} .$$
In particular, denoting $|x|^{|\beta|\gamma}D^\alpha_x D^\beta_y u$ by $ D_\gamma^\alpha u $, we will endow $W_\gamma^{2, p}(\R^N)$ with the norm
\begin{equation}\label{w2norm}
	\|u\|_{W^{2,p}_\gamma}  =  \sum_{|\alpha|=2 } \|D_\gamma^\alpha u \|_{L^p(\R^N)} +\left\|\nabla_\gamma u\right\|_{L^p(\R^N)} + \| u\|_{L^p(\R^N)}. 
\end{equation}
See  Metafune, Negro and Spina \cite{Me-N-Spi-20}.

Finally, we present our results on the existence and regularity of solutions.

\begin{theorem}\label{existencia_de_solucao} Let the assumptions of the items $(i')$ and $(ii')$ in Theorem \ref{imersao} hold, and assume \ref{V}, \ref{Q}, and \ref{(f_1)}-\ref{(f_3)}. Then problem \eqref{P} has at
	least one nonnegative nontrivial weak solution $u\in E_V^\gamma$. Furthermore,
	\begin{itemize}
		\item [$(i)$] if $\dfrac{1}{Q} \in L^\infty_{\mathrm{loc}}(\R^N)$, then $u\in L^\infty_{\mathrm{loc}}(\mathbb{R}^N)$;

		\item[$(ii)$] if there exist $V_0, V_1>0$ such that $V_0\leq V(z)\leq V_1$, for all $z\in\R^N$, then $u \in W_\gamma^{2,\frac{2^*_\gamma}{q-1}}(\R^N)$.
	\end{itemize}
\end{theorem}

\section{Proof of Theorem~\ref{imersao}}

We first prove an auxiliary result.

\begin{lemma}\label{embedding-ball}
    Let $1\leq p \leq 2^*_\gamma$, $R>0$, and $u\in E^\gamma_V$. Then, there exists $C_0:=C_0(p,R)>0$ such that
\begin{equation}\label{auxembb}
       \|u\|_{L^p(B_R)} \leq C_0 \|u\|_{E_V^\gamma} \ ,
    \end{equation}
    where $B_R$ denotes Euclidean ball 
$$B_R:=\left\lbrace z\in\mathbb{R}^N : |z|< R \right\rbrace .$$
\end{lemma}

\begin{proof}
Let $u \in E_V^\gamma$. Consider $\varphi \in C_0^{\infty}\left(B_{2R}\right)$ such that 
\begin{equation}\label{def-varphi}
    0 \leqslant \varphi \leqslant 1 \ \mbox{ in}\  B_{2R}, \ \varphi \equiv 1 \ \mbox{ in  } B_{R},  \mbox{ and  } |\nabla \varphi| \leqslant C \ \mbox{ in  } B_{2R}.
\end{equation}
In particular, $|\nabla _\gamma \varphi| \leq C_1$ in $B_{2R}$. We claim that  $\varphi u \in H_0^{1, \gamma}\left(B_{2R}\right)$.  In fact, we must prove that there exists $\left(\psi_n\right)_n \subset C_0^{\infty}\left(B_{2R}\right)$ such that $$\left\|\nabla_\gamma \psi_n-\nabla_\gamma (\varphi u)\right\|_{L^2(B_{2R})} \rightarrow 0,$$
as $n\rightarrow\infty$. By definition of $E_V^\gamma$, there exists a sequence $(u_n)_n \subset C_0^{\infty}\left(\mathbb{R}^N\right)$ satisfying $\|u_n-u\|_{E_V^\gamma} \rightarrow 0$, as $n\rightarrow \infty$. Setting $\psi_n:=\varphi u_n \in C_0^\infty(B_{2R})$, we can conclude that
$$
\begin{aligned}
    \nabla_\gamma \psi_n
-\nabla_\gamma\left(\varphi  u\right) & = [\varphi \nabla_\gamma u_n + u_n \nabla_\gamma \varphi] - [\varphi \nabla_\gamma u + u \nabla_\gamma \varphi]
\\
& = \varphi\left[\nabla_\gamma u_n-\nabla_\gamma u\right]+\nabla_\gamma \varphi\left[u_n- u\right] .
\end{aligned}
$$
Accordingly, we use the conditions on $\varphi$, the above identity, and the assumption \ref{V} to obtain
\begin{align*}
\left\|\nabla_\gamma \psi_n-\nabla_\gamma\left(\varphi u\right)\right\|_{L^2(B_{2R})}^2 & \leq \left\|\nabla_\gamma\left(u_n-u\right)\right\|_{L^2(B_{2R})}^2 + \frac{C_1^2}{\inf_{B_{2R}} V(z)}\int_{B_{2R}} V(z)|u_n-u|^2\,\mathrm{d}z \\ & \leq C_2\left\|u_n-u\right\|_{E_V^\gamma}^2  \to 0 ,
\end{align*}
as $n\rightarrow\infty$. We have proved our claim.

Now, we can apply the embedding by Kogoj and Lanconelli \cite[Theorem 3.3]{Kogoj-Lan-12} reformulated in terms of $H_0^{1,\gamma}(B_{2R})$:
\begin{equation}
H_0^{1,\gamma}(B_{2R}) \hookrightarrow L^p(B_{2R}).
\end{equation}
Thus,
\begin{equation}\label{Lp-ball-u-phi}
    \int_{B_R}|u|^p \mathrm{d}z = \int_{B_{R}}|\varphi u|^p \mathrm{d}z \leq \int_{B_{2R}}|\varphi u|^p \mathrm{d}z \leq C_3 \left(\int_{B_{2R}} |\nabla_\gamma(\varphi u)|^2 \,\mathrm{d}z \right)^{\frac{p}{2}}.
\end{equation}
So,
\begin{equation}\label{est-varphi}
\begin{aligned}
    \int_{B_R}|u|^p \mathrm{d}z & \leq C_4 \left(\int_{B_{2R}} \left[u^2\left|\nabla_\gamma \varphi\right|^2+\varphi^2\left|\nabla_\gamma u\right|^2\right] \, \mathrm{d}z \right)^{\frac{p}{2}}
    \\
    &\leq C_5 \left(\int_{B_{2R}}\left[\left|\nabla_\gamma u\right|^2 + V(z) u^2\right] \mathrm{d}z \right)^{\frac{p}{2}}  \leq C_5\|u\|_{E_V^\gamma}^p.
\end{aligned}    
\end{equation}
It proves \eqref{auxembb}.
\end{proof}

\begin{proof}[\textbf{\textit{Proof of Theorem \ref{imersao}}}]
	Observe that we can write
	\begin{equation}\label{split}
		\int_{\mathbb{R}^N}Q(z)|u|^p\,\mathrm{d}z = \int_{B_1}Q(z)|u|^p\,\mathrm{d}z + \int_{\mathbb{R}^N\setminus B_1}Q(z)|u|^p\,\mathrm{d}z .
	\end{equation}
From \ref{Q} and \eqref{auxembb} (with $R=1$), it holds that
\begin{equation}\label{ime-bola-Qp}
		\int_{B_1}Q(z)|u|^p\,\mathrm{d}z \leq C_Q\int_{B_1}|u|^p\,\mathrm{d}z \leq C_1\|u\|_{E_V^\gamma}^p.
	\end{equation}

It remains to estimate the second integral  on
	the right-hand side of \eqref{split}. To do this, for each $j\in\{0\}\cup\mathbb{N}$, define the annulus 	
	$$
	A_j = \{z\in\R^N: 2^j<|z|<2^{j+1}\}.
	$$
    Notice that
\begin{equation}\label{undersum}
		\int_{\mathbb{R}^N\setminus B_1}Q(z)|u|^p\,\mathrm{d}z = \sum_{j=0}^\infty \int_{A_j}Q(z)|u|^p\,\mathrm{d}z .
	\end{equation}
From condition \ref{Q}, we can estimate
\begin{equation*}\label{QAj}
		\int_{A_j}Q(z)|u|^p\,\mathrm{d}z 
		\leq  C_Q  \int_{A_j} \frac{1}{(1+|z|)^\beta} |u|^p\,\mathrm{d}z \\
		\leq   \frac{C_Q}{2^{\beta j}} \int_{A_j}  |u|^p\,\mathrm{d}z.
	\end{equation*}
By using the change of variables $w=2^{-j}z$, we get
	\begin{equation*}
		\int_{A_j}  |u(z)|^p\,\mathrm{d}z = \int_{A_0}  |u(2^jw)|^p\ 2^{Nj}\,\mathrm{d}w = 2^{Nj}  \int_{A_0}  |u_j(w)|^p \,\mathrm{d}w,
	\end{equation*}
where $u_j(w) := u(2^jw)$.  Hence,
	\begin{equation}\label{QAjA0}
		\int_{A_j}Q(z)|u|^p\,\mathrm{d}z  \leq   \frac{C_Q}{2^{(\beta-N) j}} \int_{A_0}  |u_j(w)|^p \,\mathrm{d}w .
	\end{equation}
Applying \eqref{est-varphi} (the first inequality), with $R=2$ and using the fact that $A_0 \subset B_2$, one has
\begin{equation}\label{imer-uj}
    \int_{A_0}  |u_j(w)|^p \,\mathrm{d}w \leq C_2\left(\int_{B_{4}} \left[\left|\nabla_\gamma u_j(w)\right|^2 + u_j^2(w)\right] \, \mathrm{d}z \right)^{\frac{p}{2}}.
\end{equation}
Recalling that if $w=(x,y)$ and $z=(\widetilde{x},\widetilde{y})=(2^jx,2^jy)= 2^jw$, the identity
$$
	\nabla_\gamma u_j(w) =  \left(2^j \nabla_x u(2^j w),2^j |x|^{\gamma} \nabla_y u(2^j w) \right) 
	$$
	gives
$$
	\begin{aligned}
		\int_{\mathbb{R}^N}  |\nabla_\gamma u_j(w) |^2 \,\mathrm{d}w &= 2^{2j} \int_{\mathbb{R}^N}   \left[ |\nabla_{x} u(2^jw)|^2 + |x|^{2\gamma} |\nabla_{y} u(2^jw)|^2\right]  \,\mathrm{d}w
		\\
		&=2^{(2-N)j}\int_{\mathbb{R}^N}   \left[ |\nabla_{x} u(z)|^2 + |2^{-j}\widetilde{x}|^{2\gamma} |\nabla_{y} u(z)|^2\right]  \,\mathrm{d}z.
	\end{aligned}
	$$
So,
	$$
	\begin{aligned}
		\int_{\mathbb{R}^N}  |\nabla_\gamma u_j(w) |^2 \,\mathrm{d}w &= 2^{(2-N)j}\int_{\mathbb{R}^N}|\nabla_{x} u(z)|^2\,\mathrm{d}z + 2^{(2-N-2\gamma)j}\int_{\mathbb{R}^N}|\widetilde{x}|^{2\gamma} |\nabla_{y} u(z)|^2\,\mathrm{d}z
		\\
		&\leq 2^{(2-N)j} \int_{\mathbb{R}^N} |\nabla_{\gamma} u(z)|^2  \,\mathrm{d}z.
	\end{aligned}
	$$
    We assume for now that $\alpha\geq0$. From \ref{V}, we have
\begin{equation}\label{A0}
 \int_{B_4}u_j^2(w) \, \mathrm{d}w = 2^{-Nj}\int_{B_{4\cdot2^j}}u^2(z) \, \mathrm{d}z \leq C_32^{(\alpha-N)j}\int_{\mathbb{R}^N} V(z)u^2(z) \, \mathrm{d}z.   
\end{equation}
    The last two inequalities together with \eqref{imer-uj}, infer that
    $$
\int_{A_0}  |u_j(w)|^p \,\mathrm{d}w \leq C_4 2^{\frac{(2-N)jp}{2}}\|u\|_{E_V}^p + C_42^{\frac{(\alpha-N)jp}{2}}\|u\|_{E_V}^p.
 $$
This combined with \eqref{undersum} and \eqref{QAjA0}, ensures that
 $$
\int_{\mathbb{R}^N\setminus B_1}Q(z)|u|^p\,\mathrm{d}z \leq C_5\|u\|_{E_V}^p \sum_{j=0}^\infty2^{\zeta_1j}+ C_5\|u\|_{E_V}^p \sum_{j=0}^\infty 2^{\zeta_2j},
$$
where
\begin{equation}\label{zetas}
		\zeta_1:=\frac{(2-N)p}{2}-(\beta-N)\quad 
		\mbox{and} \quad
		\zeta_2:=\frac{(\alpha-N)p}{2}-(\beta-N).
	\end{equation}
For $\alpha<0$, we simply notice that $\alpha$ can be replaced with $0$ in the estimate \eqref{A0}. Consequently, it will be still true that $\zeta_2<0$. Recalling $(i)-(iii)$, it follows that $\zeta_1,\zeta_2<0$, and hence
$$
\sum_{j=0}^\infty2^{\zeta_1j} < \infty \quad \mbox{and} \quad \sum_{j=0}^\infty2^{\zeta_2j}<\infty.
$$
    As a consequence,
$$
\int_{\mathbb{R}^N\setminus B_1}Q(z)|u|^p\,\mathrm{d}z \leq C_6\|u\|_{E_V^{\gamma}}^p.
$$
 Combining the last estimate with \eqref{split} and \eqref{ime-bola-Qp}, we get
	$$\int_{\mathbb{R}^N}Q(z)|u|^p\,\mathrm{d}z \leq (C_1+C_6)\|u\|_{E_V^\gamma}^p.
	$$
	Thereby, the embedding $E_V^\gamma \hookrightarrow L^p_Q(\mathbb{R}^N)$ is continuous.

To prove compactness, let $(u_n)_n \subset E_V^\gamma$ be such that $u_n \rightharpoonup 0$ in $E_V^\gamma$. We saw in the proof of Lemma \ref{embedding-ball} that $\varphi u_n  \in H^{1,\gamma}_0(B_{2R})$, where $\varphi$ is defined in \eqref{def-varphi}. In this case, it holds that $\varphi u_n  \rightharpoonup 0$ in $H^{1,\gamma}_0(B_{2R})$.  Since the embedding $H_0^{1,\gamma}(B_{2R}) \hookrightarrow L^p(B_{2R})$ is compact, for $1 \leq p <2_\gamma^*$ (see \cite[Theorem 3.3]{Kogoj-Lan-12}), we can use \eqref{Lp-ball-u-phi} to get
$$
\int_{B_R}|u_n|^p \mathrm{d}z \leq \int_{B_{2R}}|\varphi u_n|^p \mathrm{d}z = o_n(1).
$$
Here \(o_n(1) \to 0\), as \(n\to \infty\). Hence, for \(R>0\) arbitrary, it follows that
\begin{equation}\label{compact}
\int_{\mathbb{R}^N}Q(z)|u_n|^p \, \mathrm{d}z \leq C_Q\, o_n(1) + \int_{\mathbb{R}^N \setminus B_R}Q(z)|u_n|^p \, \mathrm{d}z,
\end{equation}
where we used \ref{Q}. For simplicity, we will show the item $(i')$. As $p<2_\gamma^*$, we can pick $q_1>1$ such that $2 \leq q_1 p <2_\gamma^*$. Thus, the Hölder's inequality, yields 
$$
\int_{\mathbb{R}^N \setminus B_R}Q(z)|u_n|^p \, \mathrm{d}z \leq \left(\int_{\mathbb{R}^N \setminus B_R} Q(z) \, \mathrm{d}z\right)^{\frac{1}{q_2}}\left(\int_{\mathbb{R}^N \setminus B_R} Q(z)|u_n|^{pq_1} \, \mathrm{d}z\right)^{\frac{1}{q_1}},
$$
where $\frac{1}{q_1}+\frac{1}{q_2}=1$. For $\beta>N$, we have $Q(z) \leq  C_Q(1+|z|)^{-\beta} \in L^1(\mathbb{R}^N)$. Since $(u_n)_n$ is bounded in $E_V^\gamma$, given $\varepsilon>0$, we can take $R>0$ large such that 
$$
\int_{\mathbb{R}^N \setminus B_R}Q(z)|u_n|^p \, \mathrm{d}z \leq \varepsilon.
$$
This applied in \eqref{compact} completes the proof.
\end{proof}

\section{Proof of Theorem~\ref{existencia_de_solucao}}

We split the proof of Theorem~\ref{existencia_de_solucao} into the study of the variational structure of Problem \eqref{P}, the proof of the existence of solutions, and finally the proof of the regularity results.

\subsection{Variational structure}

We introduce the energy functional $I:E_V^\gamma \to \mathbb{R}$ associated to Problem \eqref{P}  given by
$$
I(u) = \frac{1}{2}\|u\|_{E_V^\gamma}^2 - \int_{\R^N}Q(z)F(u) \, \mathrm{d} z.
$$
To show that $I$ is well defined, see that by \ref{(f_1)}-\ref{(f_2)}, for given $\varepsilon>0$, there exists $C=C(\varepsilon,q)>0$ such that
\begin{equation}\label{fcond}
	|f(s)| \leq \varepsilon|s| + C|s|^{q-1}	\quad \mbox{and} \quad |F(s)| \leq \varepsilon s^2 + C|s|^{q}, \quad \forall \, s \in \mathbb{R}.
\end{equation}
This inequality together with Theorem~\ref{imersao} infer that
\begin{equation}\label{des-F-integral}
	\int_{\R^N} Q(z)F(u) \mathrm{d}z \leq \varepsilon\int_{\R^N}Q(z)u^2\, \mathrm{d}z  + C\int_{\R^N}Q(z)|u|^q \,\mathrm{d}z<\infty.
\end{equation}
Moreover, by standard arguments one may check that $I \in C^1(E_V^\gamma,\mathbb{R})$, with
\begin{equation}\label{derivada-Gateaux}
	I'(u)\varphi = \int_{\mathbb{R}^N}\left(\nabla_\gamma u \cdot \nabla_\gamma \varphi + V(z)u\varphi\right)\,\mathrm{d}z  - \int_{\mathbb{R}^N}Q(z)f(u)\varphi\,\mathrm{d}z, \quad \forall \, u,\varphi \in E_V^\gamma.
\end{equation}
Consequently, critical points of $I$ are weak solutions for the problem \eqref{P}.

Next, we will show that $I$ satisfies the mountain pass geometry.

\begin{lemma}\label{MPGeom}
	Let the assumptions of Theorem \ref{existencia_de_solucao} hold. Then,
	\begin{enumerate}
		\item[$(a)$] there exist positive numbers $\rho$ and $\tau$ such that $I(u) \geq \tau$, whenever $\|u\|_{E_V^\gamma}=\rho$; 
		\item[$(b)$] there exists $e \in E_V^\gamma$ such that $\|e\|_{E_V^\gamma}>\rho$ and $I(e)<0$.
	\end{enumerate}
\end{lemma}

\begin{proof}
	From \eqref{des-F-integral} and Theorem \ref{imersao}, one deduces
    $$
	I(u) \geq\left(\frac{1}{2} - \varepsilon C_1\right)\|u\|_{E_V^\gamma}^2  -C_1\|u\|_{E_V^\gamma}^q = \|u\|_{E_V^\gamma}^2\left[\left(\frac{1}{2} - \varepsilon C_1\right) - C_1\|u\|_{E_V^\gamma}^{q-2}\right]. 
	$$
Choosing $0<\varepsilon<\frac{1}{2C_1}$ and recalling that $q>2$, one can easily use the above expression to obtain
	item $(a)$.
    
Let us prove $(b)$. From \ref{(f_3)}, there exist positive constants, $C_2, C_3$ such that
	\begin{equation*}\label{F}
		F(s) \geq C_2|s|^\theta-C_3 , \quad \forall \, s \in \mathbb{R}.
	\end{equation*}
    Now, let $\varphi \in C_0^\infty(\R^N)\setminus\{0\}$. By the last inequality, we have
	$$
	I(t \varphi) \leq\frac{t^2}{2}\|\varphi\|_{E_V^\gamma}^2- C_2t^\theta \int_{\mathrm{supp} \, \varphi} Q(z) |\varphi|^\theta \, \mathrm{d}z + C_3 \int_{\mathrm{supp} \, \varphi} Q(z) \, \mathrm{d}z .
	$$
Since $\theta>2$, the last estimate implies that $I(t\varphi)\rightarrow-\infty$, as $t\rightarrow\infty$. So, there exists $t_0>0$ such that $\|e\|_{E_V^\gamma}>\rho$ and $I(e)<0$, for $e = t_0\varphi$, with $t_0>0$ large, completing the proof.
\end{proof}
    
It follows from Lemma~\ref{MPGeom} that the minimax level
$$
c_{MP}:=\inf_{g\in \Gamma}\max_{0\leq t \leq 1}I(g(t)) \geq \tau>0,
$$
where $\Gamma:=\{g \in C\left([0,1], E_V^\gamma \right): g(0) = 0 \ \ \mbox{and} \ \ I(g(1))<0 \}$, is well defined.

\vspace{0,5cm}

\subsection{Weak nonnegative solution for \eqref{P}}

We say that $(u_n)_n \subset E_V^\gamma$ is a $(PS)_c$ sequence, when
\begin{equation}\label{sc}
	I(u_n) \rightarrow c\ \ \mbox{and} \ \ \|I'(u_n)\|_{(E^\gamma_V)'}\rightarrow0,\ \ \mbox{as}\ \ n\rightarrow\infty,
\end{equation}
where $(E_V^\gamma)'$ denotes the dual
space of $E_V^\gamma$. Furthermore, we say that $I$ satisfies the $(PS)_c$ condition if any $(PS)_c$ sequence admits a convergent subsequence.

Next, we will show that $I$ satisfies the $(PS)_c$ condition, for any $c \in \mathbb{R}$.

\begin{lemma}\label{PS}
	Let the assumptions of Theorem \ref{existencia_de_solucao} hold. If $(u_n)_n \subset E_V^\gamma$ is a $(PS)_c$ sequence, then $(u_n)_n$ has a convergent subsequence in $E_V^\gamma$.
\end{lemma}

\begin{proof}
	See that from \eqref{derivada-Gateaux}, \eqref{sc},  and \ref{(f_3)}, there holds
$$
	\begin{aligned}
		C_1+C_2\left\|u_n\right\|_{E_V^\gamma} &\geq I\left(u_n\right)-\frac{1}{\theta} I^{\prime}\left(u_n\right)u_n 
		\\
		&= \left(\frac{1}{2}-\frac{1}{\theta}\right)\left\|u_n\right\|_{E_V^\gamma}^2 - \int_{\mathbb{R}^N} Q(z)\left(F\left(u_n\right)-\frac{1}{\theta} f\left(u_n\right) u_n\right)
		\\
		& \geq \left(\frac{1}{2}-\frac{1}{\theta}\right)\left\|u_n\right\|_{E_V^\gamma}^2.
	\end{aligned}
	$$
Consequently, $(u_n)_n$ is bounded in $E_V^\gamma$. Thus, up to a subsequence, $u_n \rightharpoonup u$ weakly in $E_V^\gamma$.

Next, we will show that $\|u_n-u\|_{E_V^\gamma} \to 0$, as $n\to \infty$. To do this, firstly, we recall that \eqref{sc} infers that
	\begin{equation}\label{derivada-conv-zero-un-u}
		I'(u_n)(u_n-u) \rightarrow0,\quad \mbox{as}\ n\rightarrow\infty.
	\end{equation}
On the other hand,
	\begin{equation}\label{derivada-conv-zero}
		I'(u_n)(u_n-u) = \left\|u_n\right\|_{E_V^\gamma}^2-\left\langle u_n, u\right\rangle_{E_V^\gamma} - \int_{\mathbb{R}^N} Q(z) f(u_n)\left(u_n-u\right)\,\mathrm{d}z.
	\end{equation}
We claim that
	\begin{equation}\label{convergencia-funcao-f}
		\lim_{n \to \infty} \int_{\mathbb{R}^N} Q(z) f(u_n)\left(u_n-u\right)\,\mathrm{d}z  = 0.
	\end{equation}
In fact, by \eqref{fcond} and H\"older's inequality, we deduce that
	$$
	\begin{aligned}
		\int_{\mathbb{R}^N} Q(z) f(u_n)\left(u_n-u\right)\,\mathrm{d}z &\leq \varepsilon\|u_n\|_{L_Q^2(\mathbb{R}^N)} \|u_n - u\|_{L_Q^2(\mathbb{R}^N)} 
		\\
		&\quad+ C \|u_n\|_{L_Q^{r_1(q-1)}(\mathbb{R}^N)}^{q-1} \|u_n - u\|_{L_Q^{r_2}(\mathbb{R}^N)}, 
	\end{aligned}
	$$
	with $r_1,r_2>1$ and $\frac{1}{r_1} + \frac{1}{r_2}=1$. Noting that we can choose $r_1,r_2$ satisfying $r_1(q-1),r_2 \in (2,2_\gamma^*)$, by Theorem~\ref{imersao}, we have that \eqref{convergencia-funcao-f} holds. Combining \eqref{derivada-conv-zero-un-u}-\eqref{convergencia-funcao-f}, one has 
	$$
	\|u_n-u\|_{E_V^\gamma}^2 = \left\|u_n\right\|_{E_V^\gamma}^2-2\left\langle u_n, u\right\rangle_{E_V^\gamma} + \left\|u\right\|_{E_V^\gamma}^2 = o_n(1),
	$$
	where we used that $\left\langle u_n, u\right\rangle_{E_V^\gamma} = \|u\|_{E_V^\gamma}^2+o_n(1)$. The lemma is proved.	
\end{proof}

    We conclude this subsection with the following existence result.

    \begin{proposition}\label{existencia-solucao}
	Let the assumptions of Theorem \ref{existencia_de_solucao} hold. Then problem \eqref{P} has at
	least one nonnegative nontrivial weak solution.
\end{proposition}

\begin{proof}
	Observe that without loss of generality, we can assume that $f(s)=0$ for $s\leq0$, and the above results are also valid for this modified nonlinearity, again denoted by $f$. Applying Lemmas~\ref{MPGeom} and ~\ref{PS}, we obtain $u \in E_V^\gamma$ such that $I(u)=c_{MP}>0$ and $I'(u)=0$, that is, $u\neq0$ is a weak solution of \eqref{P}. Moreover, defining $u^-(z):=\max\{-u(z),0\}$, then
	$$
	0 = I'(u)u^- = -\|u^-\|_{E_V^\gamma}^2,
	$$
	and so, $u^- =0$ a.e. in $\mathbb{R}^N$. Therefore, $u \geq 0$ a.e. in $\mathbb{R}^N$, finalizing the proof.
\end{proof}

\vspace{0,5cm}

\subsection{Regularity}

Next, we apply Moser’s iteration to get that any weak solution of \eqref{P} belongs to $L^\infty_{\mathrm{loc}}(\mathbb{R}^N)$.

    \begin{proposition}
	Let the assumptions of Theorem \ref{existencia_de_solucao} hold. If $\dfrac{1}{Q}\in L^\infty_{\mathrm{loc}}(\R^N)$, then the solution obtained in Proposition~\ref{existencia-solucao} belongs to $L^\infty_{\mathrm{loc}}(\mathbb{R}^N)$.
\end{proposition}

\begin{proof}
	For each $n \in \mathbb{N}$ and $\xi>1$, we consider the sets $A_n:=\left\lbrace z \in \mathbb{R}^N : u^{\xi-1}(z) \leq n\right\rbrace $ and $D_n:= \mathbb{R}^N \setminus A_n$. 
Let
	$$
	w_n:=\left\{
	\begin{aligned}
		& u^{2\xi-1} ,\ &\mbox{in} \ A_n, \\
		&n^2u, \ &\mbox{in} \ D_n.
	\end{aligned}
	\right.
	$$
We claim that \(w_n \in E_V^\gamma\). In fact, see that
	$$
	\nabla_{\gamma}w_n=\left\{
	\begin{aligned}
		& (2\xi-1)u^{2(\xi-1)}\nabla_{\gamma}u ,\ &\mbox{in} \ A_n, \\
		&n^2\nabla_{\gamma}u, \ &\mbox{in} \ D_n.
	\end{aligned}
	\right.
	$$
    This implies
	$$
	\int_{\R^N}|\nabla_{\gamma} w_n|^2\,\mathrm{d}z \leq (2\xi-1)^2n^4\int_{A_n}|\nabla_\gamma u|^2\,\mathrm{d}z + n^4\int_{D_n}|\nabla_\gamma u|^2\,\mathrm{d}z<\infty
	$$
	and
\[
\begin{aligned}
    \int_{\R^N}V(z)w_n^2\,\mathrm{d}z &= \int_{A_n}V(z)u^{4(\xi-1)}u^2\,\mathrm{d}z + n^4\int_{D_n}V(z)u^2\,\mathrm{d}z 
    \\
    &\leq n^4\int_{\R^N}V(z)u^2\,\mathrm{d}z<\infty.
\end{aligned}
\]
Consequently, $w_n \in E_V^\gamma$ and the claim is proved.

On other hand, since $u$ is solution of \eqref{P}, then
\begin{equation}\label{solucao}
		\int_{\mathbb{R}^N}\nabla_\gamma u \cdot\nabla_\gamma w_n\,\mathrm{d}z + \int_{\mathbb{R}^N}V(z)uw_n\,\mathrm{d}z 		= \int_{\R^N}Q(z)f(u)w_n\,\mathrm{d}z 	.
	\end{equation}
    Now, consider
	$$
	v_n:=\left\{
	\begin{aligned}
		& u^{\xi} ,\ &\mbox{in} \ A_n, \\
		&nu, \ &\mbox{in} \ D_n.
	\end{aligned}
	\right.
	$$
	Hence, 
	$$
	\nabla_\gamma v_n:=\left\{
	\begin{aligned}
		& \xi u^{\xi-1}\nabla_\gamma u ,\ &\mbox{in} \ A_n, \\
		&n\nabla_\gamma u, \ &\mbox{in} \ D_n.
	\end{aligned}
	\right.
	$$
	Similarly, $v_n \in E_V^\gamma$. Notice that
	\begin{equation}\label{derivada-fraca-wn}
		\int_{\mathbb{R}^N}\nabla_\gamma u \cdot \nabla_\gamma w_n\,\mathrm{d}z = (2\xi-1)\int_{A_n}u^{2(\xi-1)}|\nabla_{\gamma} u|^2\,\mathrm{d}z + n^2\int_{D_n}|\nabla_\gamma u|^2\,\mathrm{d}z
	\end{equation}
	and
    $$
	\int_{\mathbb{R}^N}|\nabla_{\gamma} v_n|^2\,\mathrm{d}z = \xi^2\int_{A_n}u^{2(\xi-1)}|\nabla_{\gamma} u|^2\,\mathrm{d}z + n^2\int_{D_n}|\nabla_\gamma u|^2\,\mathrm{d}z.
	$$
So,
	$$
	\int_{\mathbb{R}^N}|\nabla_{\gamma} v_n|^2\,\mathrm{d}z - \int_{\mathbb{R}^N}\nabla_\gamma u \cdot \nabla_\gamma w_n\,\mathrm{d}z = (\xi^2-2\xi+1)\int_{A_n}u^{2(\xi-1)}|\nabla_{\gamma} u|^2\,\mathrm{d}z.
	$$
From this and \eqref{derivada-fraca-wn}, it holds that
	$$
	\begin{aligned}
		\int_{\mathbb{R}^N}|\nabla_{\gamma} v_n|^2\,\mathrm{d}z &\leq \frac{(\xi^2-2\xi+1)}{(2\xi-1)}\int_{\mathbb{R}^N}\nabla_\gamma u \cdot \nabla_\gamma  w_n\,\mathrm{d}z + \int_{\mathbb{R}^N}\nabla_\gamma u \cdot \nabla_\gamma w_n\,\mathrm{d}z
		\\
		& < \xi^2\int_{\mathbb{R}^N}\nabla_\gamma u \cdot \nabla_\gamma w_n\,\mathrm{d}z.
	\end{aligned}
	$$
Applying \eqref{solucao},
    $$
	\int_{\mathbb{R}^N}|\nabla_{\gamma} v_n|^2\,\mathrm{d}z \leq \xi^2\left(	 \int_{\R^N}Q(z)f(u)w_n\,\mathrm{d}z - \int_{\mathbb{R}^N}V(z)uw_n\,\mathrm{d}z 	\right).
	$$	
Combining the above inequality with the fact that $\xi>1$ and the equality
	$$
	\int_{\mathbb{R}^N}V(z)v_n^2\,\mathrm{d}z = \int_{\mathbb{R}^N}V(z)uw_n\,\mathrm{d}z,
	$$	
	we infer that
\begin{equation}\label{est-f(u)}
		\|v_n\|_{E_V^\gamma}^2 <\xi^2	 \int_{\R^N}Q(z)f(u)w_n\,\mathrm{d}z.
	\end{equation}
    From \eqref{fcond} and the estimate $w_n \leq u^{2\xi-1}$ in $\mathbb{R}^N$, one deduces
$$
	\int_{\R^N}Q(z)f(u)w_n\,\mathrm{d}z \leq \varepsilon\int_{\mathbb{R}^N}Q(z)u^{2\xi}\,\mathrm{d}z + C\int_{\mathbb{R}^N}Q(z)u^{2\xi+q-2}\,\mathrm{d}z.
	$$ 
    By H\"older's inequality,
	$$
	\begin{aligned}
		\int_{\R^N}Q(z)f(u)w_n\,\mathrm{d}z &\leq \varepsilon\left(\int_{\mathbb{R}^N}Q(z)\,\mathrm{d}z\right)^{\frac{1}{r_1}}\left(\int_{\mathbb{R}^N}Q(z)u^{2r_2\xi}\,\mathrm{d}z\right)^{\frac{1}{r_2}}
		\\
		&\quad + C\left(\int_{\mathbb{R}^N}Q(z)u^{2q_1\xi}\,\mathrm{d}z\right)^{\frac{1}{q_1}}\left(\int_{\mathbb{R}^N}Q(z)u^{q_2(q-2)}\,\mathrm{d}z\right)^{\frac{1}{q_2}},
	\end{aligned}
	$$	
	where $\frac{1}{r_1} + \frac{1}{r_2} =  \frac{1}{q_1} + \frac{1}{q_2}=1$, with 
$$
	\frac{2}{q-2}\leq q_2-1< q_2< \frac{2_\gamma^*}{q-2}.
	$$
	This is possible because $\frac{2_\gamma^*}{q-2} - \frac{2}{q-2}>1$, for any $q\in(2,2_\gamma^*)$. Since $\beta>N$, $Q\in L^1(\mathbb{R}^N)$. Moreover, $E_V^\gamma \hookrightarrow L_Q^{q_2(q-2)}(\mathbb{R}^N)$ (see Theorem~\ref{imersao}). Hence, if $q_1=r_2$, then
	$$
	\int_{\R^N}Q(z)f(u)w_n\,\mathrm{d}z \leq C_1\|u\|_{L^{2q_1\xi}_Q(\mathbb{R}^N)}^{2\xi},
	$$
	for some $C_1=C_1(\|u\|_{E^\gamma_V})>0$. This and \eqref{est-f(u)} infer that 
	$$
	\|v_n\|_{E_V^\gamma}^2 <\xi^2C_1\|u\|_{L^{2q_1\xi}_Q(\mathbb{R}^N)}^{2\xi}.
	$$
By Theorem~\ref{imersao} and the last inequality, we have
	$$
\left(\int_{A_n}Q(z)u^{2_\gamma^*\xi}\,\mathrm{d}z\right)^{\frac{2}{2_\gamma^*}}\leq\|v_n\|_{L^{2_\gamma^*}_Q(\mathbb{R}^N)}^2 \leq \xi^2C_2\|u\|_{L^{2q_1\xi}_Q(\mathbb{R}^N)}^{2\xi}.
	$$
where we used that $v_n=u^\xi$ in $A_n$. Taking the limit as $n\to \infty$ and using Fatou’s lemma, we reach
	$$
	\|u\|_{L^{2_\gamma^*\xi}_Q(\mathbb{R}^N)}^{2\xi} \leq \xi^2C_2\|u\|_{L^{2q_1\xi}_Q(\mathbb{R}^N)} 
	$$
    or
	\begin{equation}\label{inicial-ite-moser}
		\|u\|_{L^{2_\gamma^*\xi}_Q(\mathbb{R}^N)} \leq \xi^{\frac{1}{\xi}} C_2^{\frac{1}{2\xi}}\|u\|_{L^{2q_1\xi}_Q(\mathbb{R}^N)}.
	\end{equation}
Now, we start Moser's iteration, taking $\xi=\sigma= \frac{2_\gamma^*}{2q_1}>1$. Indeed,
	$$
	q_1=\frac{q_2}{q_2-1}=q_2\cdot\frac{1}{q_2-1}<\frac{2_\gamma^*}{q-2}\cdot \frac{q-2}{2} = \frac{2_\gamma^*}{2}.
	$$
Thus,
	$$
\|u\|_{L^{2_\gamma^*\sigma}_Q(\mathbb{R}^N)} \leq \sigma^{\frac{1}{\sigma}} C_2^{\frac{1}{2\sigma}}\|u\|_{L^{2_\gamma^*}_Q(\mathbb{R}^N)}.
	$$
    Taking $\xi=\sigma^2$ in \eqref{inicial-ite-moser} and noting that $2q_1\sigma^2=(2q_1\sigma)\sigma = 2_\gamma^*\sigma$, the last estimate, yields
	$$
	\|u\|_{L^{2_\gamma^*\sigma^2}_Q(\mathbb{R}^N)} \leq (\sigma^2)^{\frac{1}{\sigma^2}} C_2^{\frac{1}{2\sigma^2}}\|u\|_{L^{2_\gamma^*\sigma}_Q(\mathbb{R}^N)} \leq \sigma^{\frac{1}{\sigma}+\frac{2}{\sigma^2}}C_2^{\frac{1}{2\sigma} + \frac{1}{2\sigma^2}}\|u\|_{L^{2_\gamma^*}_Q(\mathbb{R}^N)}.
	$$
By continuing the iteration, we can conclude that
	$$
	\|u\|_{L^{2_\gamma^*\sigma^j}_Q(\mathbb{R}^N)} \leq \sigma^{\sum_{k=1}^\infty \frac{k}{\sigma^k}} C_2^{\sum_{k=1}^\infty \frac{1}{2\sigma^k}}\|u\|_{L_Q^{2_\gamma^*}(\mathbb{R}^N)}.
	$$
Since these series converge, we obtain
	$$
	\|u\|_{L^{2_\gamma^*\sigma^j}_Q(\mathbb{R}^N)} \leq C_3\|u\|_{L_Q^{2_\gamma^*}(\mathbb{R}^N)}.
	$$

    Now, consider $R>0$ arbitrary. From the hypothesis $\dfrac{1}{Q} \in L^\infty_{\mathrm{loc}}(\R^N)$, there exists $Q_R>0$ such that $Q(z) \geq Q_R$, for any $z \in B_R$. Thus,
	$$
	\int_{B_R}|u|^{2_\gamma^*\sigma^j} \mathrm{d}z \leq \frac{1}{Q_R}\int_{B_R}Q(z)|u|^{2_\gamma^*\sigma^j} \mathrm{d}z \leq \frac{1}{Q_R}\|u\|_{L^{2_\gamma^*\sigma^j}_Q(\mathbb{R}^N)}^{2_\gamma^*\sigma^j} \leq \frac{C_3^{2_\gamma^*\sigma^j}}{Q_R}\|u\|_{L^{2_\gamma^*}_Q(\mathbb{R}^N)}^{2_\gamma^*\sigma^j}.
	$$
Consequently,
	$$
	\|u\|_{L^{2_\gamma^*\sigma^j}(B_R)} \leq \frac{C_3}{Q_R^{\frac{1}{2_\gamma^*\sigma^j}}}\|u\|_{L_Q^{2_\gamma^*}(\mathbb{R}^N)}.
	$$
Taking $j \to \infty$, it holds that $u \in L^\infty(B_R)$, finalizing the proof.
\end{proof}

We will finish the proof of Theorem \ref{existencia_de_solucao} by proving item $(ii)$. First, we have an auxiliary result.

\begin{lemma}[$W_\gamma^{2, p}$ - Regularity]\label{wreg}
	Assume that $V \in L^{\infty}(\R^N)$ and there exists $V_0>0$ such that $V_0\leq V(z)$, for all $z\in\R^N$. If $p\in[2,2_\gamma^*]$ and $u\in E_V^\gamma$ is a weak solution of
    \begin{equation}\label{linearSG}
		-\Delta_\gamma u +V u = h \in L^p(\R^N),
	\end{equation}
	then $u \in W_\gamma^{2, p}\left(\mathbb{R}^N\right)$.
\end{lemma}

\begin{proof}
	Let $u \in C_0^{\infty}\left(\R^N\right)$ be a solution for \eqref{linearSG}. By Theorem 4.21 in \cite{Me-N-Spi-20}, 
    $$\|v\|_{W^{2,p}_\gamma(\R^N)} \leq C\left(\|v\|_{L^p(\R^N)} + \| \Delta_\gamma v\|_{L^p(\R^N)}\right) ,$$
    for all $v\in W^{2,p}_\gamma(\R^N)$. Hence, we have
	\begin{align*}
		\|u\|_{W^{2,p}_\gamma(\R^N)}  
		\leq & C\left\|-\Delta_\gamma u + V u \right\|_{L^p(\R^N)} + (1+C\|V\|_{L^\infty(\R^N)})\|u\|_{L^p(\R^N)} ,
	\end{align*}
    since $V\in L^\infty(\R^N)$. Then, we obtain
	\begin{equation}\label{reg1}
		\|u\|_{W^{2,p}_\gamma(\R^N)}  \leq C_1\left( \left\|u \right\|_{L^p(\R^N)} +  \|h\|_{L^p(\R^N)} \right) .
	\end{equation}

Now, if $u\in E_V^\gamma$ is a weak solution of that equation \eqref{linearSG}, the definition of \(E_V^\gamma\) allows us to take a sequence $(u_n)_{n}$ in $C_0^{\infty}\left(\mathbb{R}^N\right)$ such that $u_n \rightarrow u$ in $E_V^\gamma$, as $n\rightarrow\infty$. Next, define \(h_n \in L^p\left(\mathbb{R}^N\right)\) by
$$
	h_n:=-\Delta_\gamma u_n+V u_n, 
	$$
for each \(n \in \mathbb{N}\). Notice that \eqref{reg1} holds for $u_n$. We claim that $h_n \rightarrow h$ in $L^p\left(\mathbb{R}^N\right)$. Indeed, for every $\varphi \in C_0^\infty\left(\mathbb{R}^N\right)$, we have
\[
\left|\int_{\mathbb{R}^N}\left(h_n-h\right) \varphi \, \mathrm{d}z \right| = \left|\int_{\mathbb{R}^N} \left(\nabla_\gamma\left(u_n-u\right) \cdot \nabla_\gamma \varphi  +  V(z)\left(u_n-u\right) \varphi \right) \, \mathrm{d}z\right|.
\]
Again, using Hölder's inequality, we can conclude that
\begin{equation}\label{distest}
    \begin{aligned}
\left|\int_{\mathbb{R}^N}\left(h_n-h\right) \varphi \, \mathrm{d}z \right| \leq C_2\left\|u_n-u\right\|_{E_V^\gamma} \|\varphi \|_{W^{1, 2}_\gamma(\R^N)} \to 0,
    \end{aligned}
\end{equation}
as $n\rightarrow\infty$, for all $\varphi \in C_0^{\infty}(\mathbb{R}^N)$.

	 This means that $h_n \rightarrow h$ in the sense of the distributions.	Now, take $\left\{\varphi_l\right\}_{l\in\mathbb{N}} \subset C_0^{\infty}\left(\mathbb{R}^N\right)$, with $\varphi_l \rightarrow \varphi \in L^{p'}(\R^N)$, $\frac{1}{p} + \frac{1}{p'}=1$, and $\|\varphi\|_{L^{p'}(\R^N)}=1$. Then, \eqref{distest} and Hölder inequality give us
	\begin{align*}
		\left|\int_{\R^N} \left(h_n-h\right) \varphi \, \mathrm{d}z \right| 
		\leq  C_2\left\|u_n-u\right\|_{E_V^\gamma} \|\varphi_l \|_{W_\gamma^{1, 2}(\R^N)} + \|h_n-h\|_{L^p(\R^N)}\|\varphi_l - \varphi\|_{L^{p'}(\R^N)} .
	\end{align*}
	Since $\varphi_l \rightarrow \varphi$ in $L^{p'}(\R^N)$, given $0<\varepsilon<1$, there is $l_0 \in \mathbb{N}$ such that $\left\|\varphi - \varphi_{l_0}\right\|_{L^{p'}(\R^N)}<\varepsilon$. Then, by duality, we have
	$$
	(1-\varepsilon)\left\|h_n-h\right\|_{L^p(\R^N)} \leq C_2\left\|u_n-u\right\|_{E_V^\gamma} \|\varphi_{l_0} \|_{W_\gamma^{1, 2}(\R^N)} \rightarrow 0 ,
	$$
    as $n \rightarrow \infty$.

    As noticed by Alves and de Holanda \cite{Alves-Angelo}, we have the Sobolev embedding
    $$W_\gamma^{1, 2}(\R^N) \hookrightarrow L^p(\R^N),$$
    for $p\in[2,2_\gamma^*]$. Recall also that the assumptions on $V$ imply that the norms in $E_V^\gamma$ and $W_\gamma^{1, 2}(\R^N)$ are equivalent, and that $(u_n)_n$ and $(h_n)_n$ are uniformly bounded in $E_V^\gamma$ and $L^p(\R^N)$, respectively. These facts with \eqref{reg1}, and the aforementioned Sobolev embedding give
	\begin{equation}\label{reg2}
		\left\|u_n\right\|_{W_\gamma^{2, p}(\R^N)} \leq C_3\left( \left\|u_n \right\|_{W_\gamma^{1, 2}(\R^N)} +  \|h_n\|_{L^p(\R^N)} \right) \leq C_4,
	\end{equation}
	uniformly on $n\in\mathbb{N}$. Therefore, from the reflexivity of $W_\gamma^{2, p}(\R^N)$, there exists a subsequence $u_{n_i} \rightharpoonup \tilde{u}$ in $W_\gamma^{2, p}(\R^N)$, as $i\rightarrow \infty$. By the uniqueness of the weak limit in  $L^p(\R^N)$, we obtain that $u=\tilde{u}$ and 
	\begin{equation*}
		\|u\|_{W_\gamma^{ 2,p}} \leq  C_4,
	\end{equation*}
    completing the proof.
\end{proof}

It only remains to establish part (ii) of Theorem \ref{existencia_de_solucao}.

\begin{proof} 
We will show that \(Q f(u) \in L^s(\mathbb{R}^N)\), with $s=\frac{2^*_\gamma}{q-1}$. In fact, from \eqref{fcond} and the estimate $(a+b)^s\leq C_s(a^s + b^s)$, we obtain
\[
\int_{\R^N} \left|Q(z) f(u)\right|^s \, \mathrm{d} z \leq \varepsilon^s \int_{\R^N} Q^s(z)|u|^s\, \mathrm{d} z + C^s\int_{\R^N} Q^s(z)|u|^{s(q-1)} \, \mathrm{d} z.
\]
On the one hand, Hölder's  inequality together with \ref{Q} gives
\[
\int_{\R^N} Q^s(z)|u|^s\, \mathrm{d} z \leq C_Q^{s-1}\|Q\|_{L^1(\R^N)}^{1-\frac{s}{2^*_\gamma}} \|u\|_{L_Q^{2^*_\gamma}(\R^N)}^s.
\]
On other hand, by \ref{Q}, it holds that
\[
\int_{\R^N} Q^s(z)|u|^{s(q-1)} \, \mathrm{d} z \leq C_Q^{s-1}\|u\|_{L_Q^{2^*_\gamma}(\R^N)}^{2_\gamma^*}.
\]
The last three estimates combined with Theorem \ref{imersao}, imply that 
\[
\int_{\R^N} \left|Q(z) f(u)\right|^s \, \mathrm{d} z<\infty,
\]
where we used that \(Q \in L^1(\mathbb{R}^N)\), because \(\beta>N\). Therefore, Lemma \ref{wreg} gives the result.
\end{proof}

\section*{Acknowledgments} 

We thank Prof. A. Kogoj for many fruitful discussions about this problem.

J. Carvalho has been supported by by CNPq with grant 300853/2025-4.

A. Viana is partially CAPES-Humboldt Research Fellowship PE32739979, and is part of Universal Fapitec Project n° 019203.01303/2024-1.

Part of this research was conducted while A. Viana was visiting the Abdus Salam International Centre for Theoretical Physics (ICTP) and the \textit{Institut f\"ur Angewandte Mathematik, Universit\"at Ulm}. He would like to thank these institutions for the hospitality.

\end{document}